\documentclass[12pt,a4paper, twoside, reqno]{amsart} %Change as appropriate

\usepackage[margin=3cm]{geometry}

\usepackage{ifpdf}
\usepackage{graphicx}
\usepackage[T1]{fontenc} % Font encoding.
\usepackage[latin1]{inputenc} % Input encoding - change for other
\usepackage{amsmath} % Extra environments etc.
\usepackage{amsthm} % Nice proof environment.
\usepackage{amsfonts} % Additional fonts available. % is autoloaded by amssymb
\usepackage{amssymb}
\usepackage{amsbsy}
\usepackage{bbm} % More fonts
\usepackage[english]{babel} % Multiple languages support.
\usepackage{varioref} % A few fancy references
\usepackage{enumerate} % Better enumerate environment
 \usepackage{setspace}
\makeatother

\theoremstyle{definition}

\newtheorem*{defn*}{Definition}

\theoremstyle{remark}

\theoremstyle{plain}
\newtheorem{thm}{Theorem}[section]
\newtheorem*{thm*}{Theorem}

\newtheorem{prop}[thm]{Proposition}
\newtheorem{cor}[thm]{Corollary}

\newtheorem*{thmDS}{Duffin-Schaeffer Conjecture (1941)}

\newtheorem*{thmBDV1}{Theorem BDV1 (2006)}
\newtheorem*{thmBDV2}{Theorem BDV2 (2006)}
\newtheorem*{thmSchl}{Schl\"{o}milch's Theorem (Late 19th Century)}

\newcommand{\norm}[1]{\ensuremath{\left\Vert #1 \right\Vert}}

\newcommand{\infabs}[1]{\ensuremath{\left\vert #1 \right\vert}}

\numberwithin{equation}{section}

\title{On a mixed Khintchine problem in Diophantine approximation}

\author{Stephen Harrap}

\address{S. Harrap, Department of Mathematics, Aarhus University,
 Ny Munkgade 118, BLG.1530, DK-8000 Aarhus C, Denmark}
  
\author{Tatiana Yusupova}
 
\address{T. Yusupova, Department of Mathematics, University of York,
 Heslington, York YO10 5DD, United Kingdom}

\date{}

\begin{document}

\begin{abstract}
We establish a `mixed' version of a fundamental theorem of Khintchine \cite{Khi} within the field of simultaneous Diophantine approximation. Via the 
notion of `ubiquity' we are able to make significant progress towards the completion of the metric theory associated with mixed problems in this setting. This includes finding a natural mixed analogue of the classical Jarn\'ik-Besicovich Theorem. Previous knowledge surrounding mixed problems in this area was almost entirely restricted to the multiplicative setup of de Mathan \& Teuli\'e \cite{dMT}, where the concept originated.
\end{abstract}

%\onehalfspace
\maketitle
\section{Introduction}
\label{sec:introduction1}

A classical result of Dirichlet implies that for any real vector $\mathbf{x}=(x_1, x_2) \in [0, 1)^2$ the system of inequalities
\begin{equation*}
	\infabs{x_1-\frac{p_1}{q}} \, \leq \, \frac{1}{q^{3/2}},  \quad \quad \quad 
	\infabs{x_2-\frac{p_2}{q}} \, \leq \, \frac{1}{q^{3/2}}
\end{equation*}
is satisfied for infinitely many $p_1, p_2 \in \mathbb{Z}$ and $q \in \mathbb{N}$.
This tells us that any real vector can be approximated by rational vectors $(p_1/q, p_2/q)$ at a `rate' of $q^{-3/2}$. Moreover, this rate can be considered optimal in the obvious sense.
This idea can be generalised in the following manner.
Choose any positive real numbers $i$ and~$j$ satisfying
\begin{equation}
	\label{eqn:ijconditions}
	i, \, j \, > \, 0 \quad \text{ and } \quad i+j \, = \, 1 
\end{equation}
and let $\psi : \mathbb{N} \rightarrow \mathbb{R}_{\geq0}$ be any non-negative arithmetic function. 
For reasons that will become apparent we refer to $\psi$ as an \textit{approximating function}.
Consider the set $W(i, j, \psi)$ of real vectors $\mathbf{x}=(x_1, x_2) \in [0,1)^2$ for which the system of inequalities
\begin{equation}
	\label{eqn:inequalities}
	\infabs{x_1-\frac{p_1}{q}} \, \leq \, \frac{\psi^i(q)}{q},  \quad \quad \quad 
	\infabs{x_2-\frac{p_2}{q}} \, \leq \, \frac{\psi^j(q)}{q}
\end{equation}
is satisfied for infinitely many $p_1, p_2 \in \mathbb{Z}$ and $q \in \mathbb{N}$.  Essentially, if a vector $\mathbf{x}$ is contained in $W(i, j, \psi)$ then  it can be approximated by rational points at a rate prescribed by the approximating function~$\psi$.
The exponents $i$ and $j$ act as `weights', perturbing the speed of approximation across the two components of $\mathbf{x}$.
Evidently, the set $W(i, j, \psi)$ is only interesting if the function $\psi$ takes small values for large~$q$. It is therefore reasonable to assume, as we will here,  
that $\psi(q) \rightarrow 0$ as $q \rightarrow \infty$.

In 1926, Khintchine \cite{Khi} proved a remarkable result describing the rate at which real vectors can `typically' be approximated by rational points in the case `$i=j=1/2$'. 
Khintchine's result was later generalised by Schmidt~\cite{Sch}, who proved that for any pair of real numbers $i, j$ satisfying (\ref{eqn:ijconditions}) and any approximating function $\psi$ we have
	\begin{equation*}
		\lambda_2 \left( W(i, j, \psi) \right) \quad = \quad
		\begin{cases}
			0, & \displaystyle\sum_{r=1}^{\infty} \psi(r) \, < \, \infty. \\
  		& \\
  		1, & \displaystyle\sum_{r=1}^{\infty} \psi(r) \, = \, \infty\, \text{ and } \psi \text{ is monotonic}.
  	\end{cases}
	\end{equation*}
Here and throughout, $\lambda_n$ denotes standard $n$-dimensional Lebesgue measure. 
It was subsequently shown that the monotonicity restriction imposed on $\psi$ in the `divergent' part of the above statement can be relaxed. For example,  Theorem~$3.8$ of Harman's book \cite{Harman} suffices. This provides a stark contrast to the present situation in the classical one dimensional setting, as we discuss in~\S\ref{sec:DuffinAndSchaeffer}.

%The approximating function $\psi$ defines a sequence of 
%neigbourhoods of rational points $(p_1/q, \, p_2/q)$ that describe the quality of approximation. In the classical case these domains of error are precisely square neighbourhoods in $\mathbb{R}^2$.
%When the neigbourhoods of rational points defined by the approximation function are genuinely rectangular the methods of \cite{BV10} can, with a bit of work, be modified 
%to reach the same conclusion. We remark that in either case the 
%neighbourhoods are convex subsets of $\mathbb{R}^2$.

The system of inequalities given by (\ref{eqn:inequalities}) can be rephrased in more compact notation to read
\begin{equation}
	\label{eqn:inequalities2}
	\max \left\{ \norm{qx_1}^{1/i}, \, \norm{qx_2}^{1/j} \right\} \, \leq \, \psi(q),  
\end{equation}
where $\norm{\, . \, }$ denotes the distance to the nearest integer. When $i=j=1/2$ the left hand side of (\ref{eqn:inequalities2}) reduces to the familiar supremum norm.
 In this case one can consider the following `multiplicative' variant of the set $W(1/2, 1/2, \psi)$ conceived by 
 replacing the supremum norm with the geometric mean. In particular, we may define
\begin{equation*}
	M(\psi):= \,  \left\{\mathbf{x} \in [0,1)^2: \, \norm{qx_1}\norm{qx_2} \,
	\leq  \, \psi(q)
	\, \, \text{ for inf.\ many } q \in \mathbb{N} \right\}.
\end{equation*}
A criterion for the Lebesgue measure of the set $M(\psi)$,  analogues to Khintchine result, was found by Gallagher~\cite{Gal} 
in $1962$. For any approximating function $\psi$ we have that
	\begin{equation*}
		\lambda_2 \left( M(\psi) \right) \quad = \quad
		\begin{cases}
			0, & \displaystyle\sum_{r=1}^{\infty}  \psi(r)\log(r) \, < \, \infty. \\
  		& \\
  		1, & \displaystyle\sum_{r=1}^{\infty} \psi(r)\log(r) \, = \, \infty\, \text{ and } \psi 
  		\text{ is monotonic}.
  	\end{cases}
	\end{equation*}
\noindent
%Here, the error domains concerned are no longer convex, they are `hyperbolic' in shape. 
It is an open question as to whether the monotonicity assumption can be safely removed from this statement (for recent progress, see \cite{BHV1}). We remark that the concept of (consciously) adding weights to the components of approximation is not prevalent in the study of problems in the multiplicative setting and we avoid its inclusion here for the sake of clarity.

 In 2004, de Mathan \& Teuli\'{e} \cite{dMT} introduced a closely related `mixed' multiplicative setup realised by retaining the condition that $\norm{qx_1}$ is small but replacing the condition on $\norm{qx_2}$ with a condition of divisibility. 
To elaborate we require some notation. A sequence $\mathcal{D}=\left\{n_k\right\}_{k=0}^\infty$ of positive integers is said to be a \textit{pseudo-absolute value sequence} if it is 
strictly increasing with $n_0 = 1$ and $n_k | n_{k+1}$ for all $k$. In this case $\mathcal{D}$ will often be referred to as a \textit{$\mathcal{D}$-adic sequence}. We say a $\mathcal{D}$-adic sequence sequence has \textit{bounded ratios} if the quotients $n_{k+1}/n_k$ do not exceed some universal constant. 

Given a $\mathcal{D}$-adic sequence we define the \textit{$\mathcal{D}$-adic pseudo-absolute value} $\infabs{ \, . \, }_\mathcal{D}: \mathbb{N} \rightarrow \left\{ 1/n_k: \, k\in \mathbb{N} \right\}$ by
\[
	\infabs{q}_\mathcal{D} := 1/n_{\omega_\mathcal{D}(q)} = \inf\{ 1/n_m : q \in  n_m\mathbb{Z} \}.
\] 
In other words, the $\mathcal{D}$-adic value assigns to each natural number $q$ the reciprocal of the largest member of $\mathcal{D}$ dividing $q$.
When $\left\{n_{k+1}/n_k\right\}_{k=0}^{\infty}$ is the constant sequence equal to a prime number $p$, the pseudo absolute value  $|\, \cdot \, |_\mathcal{D}$ is 
the usual $p$-adic absolute value $|\, \cdot \, |_{p}$.

Within the setup of de Mathan \& Teuli\'e one may consider a `mixed' version of the set $M(\psi)$. Namely, we define
\begin{equation*}
	M_{\mathcal{D}}(\psi):= \,  \left\{x \in [0,1): \, \infabs{q}_{\mathcal{D}}\norm{qx} \,	\leq  \, \psi(q)
	\, \, \text{ for inf.\ many } q \in \mathbb{N} \right\}.
\end{equation*}
Recently, in \cite{HH}, an analogue of Gallagher's statement was established concerning the set $M_{\mathcal{D}}(\psi)$. For any approximating function $\psi$ and any $\mathcal{D}$-adic sequence with bounded ratios we have
	\begin{equation*}
		\lambda_1 \left( M_\mathcal{D}(\psi) \right) \quad = \quad
		\begin{cases}
			0, & \displaystyle\sum_{r=1}^{\infty}  \psi(r)\log(r) \, < \, \infty. \\
  		& \\
  		1, & \displaystyle\sum_{r=1}^{\infty} \psi(r)\log(r) \, = \, \infty\, \text{ and } \psi 
  		\text{ is monotonic}.
  	\end{cases}
	\end{equation*}
Again, it is currently unknown whether the monotonicity assumption is necessary.
Somewhat surprisingly, a mixed analogue of the set $W(i, j, \psi)$ has not yet 
been studied. The intentions of the present paper are to do exactly that. In particular, a 
metric theorem is established concerning the one-dimensional set
\begin{equation*}
	W_{\mathcal{D}}(i, j, \psi):= \,  \left\{ x \in [0,1): \, \max \left\{\infabs{q}_{\mathcal{D}}^{1/i}, \norm{qx}^{1/j} \right\} \leq  \psi(q)
	\, \, \text{ for inf.\ many } q \in \mathbb{N} \right\}.
\end{equation*}
As we have seen, for each monotonic approximating function $\psi$ the Lebesgue measures of the multiplicative sets $M(\psi)$ and $M_{\mathcal{D}}(\psi)$ 
depend on the asymptotic behaviour of the same sum (assuming that $\mathcal{D}$ has bounded ratios). We show that the sets $W(i, j, \psi)$ and 
$W_{\mathcal{D}}(i, j, \psi)$ enjoy a similar property.

%The only progress previously made in this mixed, weighted and simultaneous setting is that presented in \cite{BLV}. 
For the case when $\psi(q)=1/q$ and $\mathcal{D}$ has bounded ratios the `badly approximable' complement of the 
set $W_{\mathcal{D}}(i, j, \psi)$ was examined in \cite{BLV} (see also~\cite{Li}). 
This seems to constitute all previous knowledge of mixed problems in Diophantine approximation outside of the multiplicative setting. On the other hand, when $\psi(q)=1/q$ the sets $M(\psi)$ and $M_{\mathcal{D}}(\psi)$ are strongly related to the famous Littlewood Conjecture and its mixed counterpart, both of which have received much recent attention.
%and various related results established. 

%Further inspiration came from the recent result of Badziahin, Levesley \& Velani \cite{BLV} stating that the set
%\begin{equation*}
%	[0,1) \setminus W_{\mathcal{D}}(i, j, \psi_{{\epsilon}}: q\mapsto \epsilon q^{-1}) \, =  \, \left\{x \in [0,1): \, \liminf_{q \rightarrow \infty} q
%	\max \left\{ \infabs{q}_\mathcal{D}^{1/i}, \,  \norm{qx}^{1/j}\right\} > \epsilon \right\}
%\end{equation*}
%is of zero Lebesgue measure $\lambda_1$ and full Hausdorff dimension for any $\epsilon>0$.

%We would like to prove a Khintchine type theorem for all non-increasing function $\psi$; i.e., 
%we would like to show the Lebesgue measure of $W_{\mathcal{D}}(i, j, \psi)$ depends only upon the convergence or divergence  of a certain sum relating to the approximating function $\psi$. This sum will necessarily depend on the set $\mathcal{A}$. There will be no loss in generality in assuming throughout that all real numbers $x$ are confined to the unit interval, since all sets and arguments considered will be invariant under integer translation. Accordingly, if $i=0$ for example, the
%set $W_{\mathcal{D}}(0, 1, \psi)$ will be identified with $W(\psi)\cap [0,1]  $, where $W(\psi)$ is the classical set of $\psi$-approximable numbers;
%\begin{equation*}
%	W(\psi):=\left\{x \in \mathbb{R}: \, \norm{qx} < \psi(q) \quad	\text{ for inf. %many } q \in \mathbb{N} \right\}.
%\end{equation*}

\section{Statement of Results}
\label{sec:StatementOfResults}

For notational purposes, let $\mathcal{A}_\psi:=\mathcal{A}(\mathcal{D}, \psi, i):= \, \left\{r \in \mathbb{N}: \, \infabs{r}_{\mathcal{D}} < \psi^i(r) \right\}$. 
%Also, we will often impose a condition the approximating function $\psi$ in the `divergent' parts of our theorems. First, notice that if $\sum_{r\in\mathbb{N}}\psi(r) = \infty$ then there exist infinitely many natural numbers $r_k$ for which 
%$\psi(r_k) > 1/r_k^{1/i}$. The following condition concerns the rate of growth of such numbers.
%
%\noindent
%$(\textbf{P}) \quad $ There exists a constant $P>1$ and a sequence of natural numbers $\left\{r_k\right\}$ s.t.
%\[
%	\psi(r_k) > 1/r_k^{1/i} \quad \text{ for which }\quad r_{k+1}/r_k \leq P\quad \forall \, k\in \mathbb{N}. 
%\]
%
%\noindent
The main result of this paper is the following analogue of the statement of Schmidt described above.

\begin{thm}
	\label{thm:divergence}
	For any pair of reals $i$, $j$ satisfying (\ref{eqn:ijconditions}), any decreasing approximating function $\psi$ and any $\mathcal{D}$-adic sequence with bounded ratios we have
	\begin{equation*}
		\lambda_1 \left( W_{\mathcal{D}}(i, j, \psi) \right) \quad = \quad
		\begin{cases}
			0, & \displaystyle\sum_{r\in \mathbb{N}} \psi(r) \, < \, \infty. \\
  		& \\
  		1, & \displaystyle\sum_{r\in \mathbb{N}} \psi(r) \, = \, \infty.
  	\end{cases}
	\end{equation*}
\end{thm}

We remark that under the assumption that $\psi$ is monotonic it is easy to show that the sums $\sum_{r\in \mathbb{N}} \psi(r)$ and $\sum_{r\in \mathcal{A}_\psi} \psi^j(r)$ are asymptotically equivalent. As a consequence, one is free to replace the former sum with the latter in the statement of Theorem \ref{thm:divergence}. This feature is symptomatic of the fact that the problem at hand is essentially one of Diophantine approximation with restricted denominator. Indeed, we may write
\begin{equation*}
	W_{\mathcal{D}}(i, j, \psi)=\left\{x \in [0, 1): \, \norm{qx} < \psi^j(q) \quad	\text{ for inf. many } q \in \mathcal{A}_\psi \right\}.
\end{equation*}

When $\psi$ is not assumed monotonic the two sums described above are not necessarily asymptotically equivalent (see \S\ref{sec:TheMixedSetting}). Moreover, using standard techniques it quickly follows that 
\begin{equation}
\label{eqn:convergence}
		\lambda_1 \left( W_{\mathcal{D}}(i, j, \psi) \right) = 0  \quad \text{ if } \quad \displaystyle\sum_{r\in \mathcal{A}_\psi} \psi^j(r) \, < \, \infty. 
	\end{equation}
However, as demonstrated in  \S\ref{sec:TheMixedSetting}, the analogous statement does not hold in general for the sum $\sum_{r\in \mathbb{N}} \psi(r)$. In this sense the monotonicity assumption is necessary in the 'convergence' part of Theorem~\ref{thm:divergence} (in its stated form).
The point is that $\sum_{r\in \mathcal{A}_\psi} \psi^j(r)$ should be considered the `genuine' critical sum relating to the measure of the set $W_{\mathcal{D}}(i, j, \psi)$. In fact, it would not be unfair to view it as merely a coincidence that there is asymptotic equivalence between the two sums in question when monotonicity is enforced. That said, to bring the similarity with Schmidt's classical result to the forefront we choose to present our statement in the current form, despite the possibly artificial nature of doing so. Regardless of which sum one chooses we  prove in \S\ref{sec:ConstructionOfCounterexample} that the monotonicity assumption in the `divergent' part of Theorem~\ref{thm:divergence} is absolutely necessary. 

%This is discussed further in \S\ref{sec:EquivalenceOfSums}.
%It would be desirable to remove the requirement of property $\textbf{P}$ in the divergent part of 
%Theorem \ref{thm:divergence}. We only mention that this property is not particularly restrictive, and is a consequence of the methods we have presented. Morover, it can be seen that when $j>i$ property 
%$\textbf{P}$ is redundant. To see this, assume that all possible sequences $\left\{r_k\right\}$ with 
%$\psi(r_k) > 1/r_k^{1/i}$ are such that $r_{k+1}/r_k \rightarrow \infty$ as $k \rightarrow \infty$ and
%$\psi^i(q) \leq 1/q$ for all $q \in \mathbb{N} \setminus \left\{r_k\right\}_{k \in \mathbb{N}}$. Then 
%there exists $k_0 \in \mathbb{N}$ such that for all $k \geq k_0$ we have $r_{k+1}-r_k>1$. Since 
%$\infabs{q}_{\mathcal{D}} \geq 1/q$ for all $q\in\mathbb{N}$ we then have 
%\begin{eqnarray*}
%	\sum_{q \in \mathcal{A}} \psi^j(q) \quad \ll \,\sum_{\substack{k \, \geq \, k_0: \\ 
%	\infabs{r_k}_{\mathcal{D}} \, 
%	< \, \psi^i(r_k)}} \: \psi^j(r_k) & \leq & \sum_{\substack{k \, \geq \, k_0: \\ \infabs{r_k}_{\mathcal{D}} \, 
%	< \, \psi^i(r_k)}} \: \psi^j(r_k-1) \\
%	& \leq & \sum_{\substack{k \, \geq \, k_0: \\ \infabs{r_k}_{\mathcal{D}} \, 
%	< \, \psi^i(r_k)}} \: (r_k-1)^{-j/i},\
%\end{eqnarray*} 
%which certainly converges when $j>i$.

It is worth emphasising that the two degenerate cases `$i=0$' and `$j=0$' are not considered in this paper. 
On employing the convention that $x^{1/y}=0$ when $y=0$ for all real $x$,
it is easily verified that in the former case Theorem \ref{thm:divergence} reduces to the classical one-dimensional result of Khintchine (see \S\ref{sec:DuffinAndSchaeffer}), whilst in the latter case the measure of the set $W_{\mathcal{D}}(1, 0, \psi)$ 
trivially fulfils a `zero-one' law. Indeed,
\[
	 W_{\mathcal{D}}(1, 0, \psi) \quad = \quad
	\begin{cases}
		[0,1), &  \psi(q)>\infabs{q}_\mathcal{D} \text{ for infinitely many } q \in \mathbb{N}. \\
		\emptyset, &  \text{otherwise}.
  \end{cases}
\]

Finally, we remark that whilst it would be desirable to generalise Theorem~\ref{thm:divergence} to the case of pseudo-absolute value sequences whose ratios are not bounded, to do so would require more than trivial improvements over the techniques we present. 

Theorem \ref{thm:divergence} is proven as a consequence of a more general Hausdorff measure result. Throughout, $\mathcal{H}^s$ denotes standard $s$-dimensional Hausdorff measure and `$\dim$' represents Hausdorff dimension. Recall that when $s=1$ Hausdorff measure is comparable to 
one-dimensional Lebesgue measure.
\begin{thm}
	\label{thm:Hdivergence}
	Fix any pair of reals $i$, $j$ satisfying (\ref{eqn:ijconditions}), any $\mathcal{D}$-adic sequence with bounded ratios and any real $s \in (i,1]$.  Then, for any approximating function $\psi$ for which the related function $f_{\psi}:\mathbb{N} \rightarrow \mathbb{R}_{\geq0}: r \longmapsto r^{1-s}\psi^{i+js}(r)$ is decreasing we have 
	\begin{equation*}
		\mathcal{H}^s \left( W_{\mathcal{D}}(i, j, \psi) \right) \, = \,
		\begin{cases}
			0, & \displaystyle\sum_{r\in \mathbb{N}} f_{\psi}(r) \, < \, \infty. \\
  		& \\
  		\mathcal{H}^s([0,1)), & \displaystyle\sum_{r\in \mathbb{N}} f_{\psi}(r) \, = \infty
  		\, \text{ and } \psi \text{ is monotonic}.
  	\end{cases}
\end{equation*}
\end{thm}
%Once more the assumption of monotonicity of $\psi$ allows the volume sum in Theorem~\ref{thm:Hdivergence} to be replaced with a more natural equivalent sum; namely, $\sum_{r\in \mathcal{A}} r^{1-s}\,\psi^{js}(r)$.
We prove this result via the notion of ubiquitous systems, a fundamental tool for establishing measure theoretic statements in Diophantine approximation. A tailored account of the ubiquity setup is presented in \S\ref{sec:Theubiquitysetup}.

We do not claim the conditions imposed on the function $\psi$ in Theorem~\ref{thm:Hdivergence} are optimal. In fact, we suspect that the monotonicity assumption imposed on $f_{\psi}$ may be unnecessary. 
Despite this constraint, we are able to deduce the following generalisation of a classical theorem of Jarn\'ik \cite{Jar} and Besicovich~\cite{Bes}. Their 
fundamental result corresponds to the case `$i=0, j=1$' in our setup.
\begin{cor}
	\label{cor:Hdimdivergence}
		Choose any pair of reals $i$, $j$ satisfying (\ref{eqn:ijconditions}), 
		any $\mathcal{D}$-adic sequence with bounded ratios and any 
		decreasing approximating function $\psi$. Assume there exists a real number 
		$\tau$ such that $$\tau \, = \, \lim_{r \rightarrow \infty}\frac{-\log \psi(r)}{\log r} \, < \,  \frac{1}{i}.$$ 
		Then, 
	\begin{equation*}
		\dim\left( W_{\mathcal{D}}(i, j, \psi) \right)  \, = \, 
		\min \left\{1, \, \frac{2 - i\tau}{1 + j\tau} \right\}.
	\end{equation*}
\end{cor}
It should be observed that if $\psi(r)\leq r^{-1/i}$ for all sufficiently large $r$ then the set $W_{\mathcal{D}}(i, j, \psi)$ is empty.

%In section \S\ref{sec:Inhom} we discuss the inhomogeneous analogue to Theorem \ref{thm:divergence} (NOWHERE NEAR COMPLETE).

\section{Removing monotonicity}
\label{sec:DS}

\subsection{The work of Duffin and Schaeffer}
\label{sec:DuffinAndSchaeffer}

For any approximating function $\psi$ let
\[
	W(\psi):= \left\{x \in [0,1): \, \norm{qx} < \psi(q) \text{ for infinitely many } q \in \mathbb{N} 
	\right\}
\]
denote the standard set of $\psi$-approximable numbers. 
The one-dimensional version of Khintchine's theorem states that the Lebesgue measure of $W(\psi)$ 
is zero if the sum $\sum_{r=1}^{\infty} \psi(r)$ converges or one if the sum  $\sum_{r=1}^{\infty} \psi(r)$ diverges and $\psi$ is decreasing.
In their seminal paper \cite{DS}, Duffin \& Schaeffer produced a counterexample showing that the monotonicity assumption in the `divergent' part  is necessary. In particular, 
they constructed a non-monotonic approximating function $\psi$ for which 
$\lambda_1(W(\psi))=0$ but the sum $\sum_{r =1}^{\infty} \psi(r)$ diverges. 

In an attempt to provide an alternative statement to that of Khintchine, free from any restrictions on the choice of approximating function, Duffin \& Schaeffer considered the following variant of $W(\psi)$.
For any approximating function $\psi$ let $W'(\psi)$ denote the set of real numbers $x \in [0,1)$ for which
\[
	\infabs{qx-p} < \psi(q) \quad \quad \text{ for infinitely many } p \in \mathbb{Z} \text{ and } q \in \mathbb{N} \text{ with } (p,q)=1. 
\]
This set differs from $W(\psi)$ only by the coprimality restriction on $p$ and $q$. The restriction ensures that the rational approximations $p/q$ to $x$ are in reduced form.  In their paper, Duffin \& Schaeffer were able to establish partial results concerning the measure of the set~$W'(\psi)$ but fell short of proving their now famous conjecture. In what follows $\varphi$ denotes Euler's totient function.

\begin{thmDS}
	For any approximating function $\psi$ we have
	\[
		\lambda_1(W'(\psi)) = 1  \quad \text{ if } \quad \sum_{r=1}^{\infty} \frac{\varphi(r)}{r}\, \psi(r) = 
		\infty.
	\]
\end{thmDS}

It is clear that $W'(\psi) \subset W(\psi)$, which immediately
implies the complementary statement 
	\[
		\lambda_1(W'(\psi)) = 0  \quad \text{ if } \quad \sum_{r=1}^{\infty} \frac{\varphi(r)}{r}\, \psi(r) < 
		\infty
	\]
holds for every approximating function $\psi$. The Duffin-Schaeffer Conjecture 
represents one of the most profound unsolved problems in metric Diophantine approximation. For a thorough account, including recent progress made concerning the conjecture, see \S$2$ of \cite{Harman} and \cite{BHHV}.

\subsection{The mixed setting}
\label{sec:TheMixedSetting}

One might expect that similar properties to those suggested by Duffin \& Schaeffer hold within the mixed simultaneous setting. In this section we 
discuss the necessity of the monotonicity assumption imposed in Theorem~\ref{thm:divergence}. Firstly, we present a simple example of a function demonstrating that the assumption is necessary in the convergence part of the result.

For any real pair $i, j$ and any pseudo-absolute value sequence $\mathcal{D}=\left\{n_k\right\}_{k=0}^\infty$ let 
\begin{equation*}
		\psi_0(q):\, = \,	
		\begin{cases}
			k^{-(1+\delta)}, & q = n_k, \\
  		& \\
  		0, & \text{otherwise},
  	\end{cases}
\end{equation*}
where $\delta:  =  (1/\max(i, j) - 1)/2 > 0$. Then, for all $k \in \mathbb{N}$ we have $n_k \geq k > k^{i(1+\delta)}$ implying that 
$\infabs{n_k}_\mathcal{D} = 1/n_k < \psi_0^i(n_k)$ and so $\mathcal{A}_{\psi_0}=\mathcal{D}$.  Moreover, it is easy to see that
$$ \displaystyle\sum_{r=1}^{\infty} \psi_0(r) \, < \, \infty \quad \text{ but } \quad W_{\mathcal{D}}(i, j, \psi_0) = [0, 1)$$
as required. Of course, statement (\ref{eqn:convergence}) provides the relevant monotonicity-free criterion giving measure zero and is certainly not 
contravened here as $$\displaystyle\sum_{r\in \mathcal{A}_{\psi_0}} \psi_0^j(r) \, = \, \sum_{r\in \mathcal{D}} \psi_0^j(r) \,= \, \sum_{k=1}^{\infty}k^{-j(1+\delta)} \, > \,  \sum_{k=1}^{\infty}k^{-1},$$
which diverges.

It is easier still to show that the monotonicity assumption is necessary in the divergence part of Theorem~\ref{thm:divergence}.
For example, one may take
\begin{equation*}
		\psi_1(q): \, = \,	
		\begin{cases}
			1/2, & (n_k, q)=1 \, \text{ for all } k \in \mathbb{N}, \\
  		& \\
  		0, & \text{otherwise},
  	\end{cases}
\end{equation*}
in which case the set $W_{\mathcal{D}}(i, j, \psi_1)$ is empty but the sum $\sum_{r =1}^{\infty}\psi_1(r)$ diverges.

The simple nature of the examples $\psi_0$ and $\psi_1$ is indicative of the fact that the volume sum in question is not the morally correct choice. As discussed earlier, a more interesting problem is removing monotonicity from Theorem~\ref{thm:divergence} when the critical sum is taken to be $\sum_{r \in \mathcal{A}_{\psi}}\psi^j(r)$. Here also, one can provide a counterexample in the divergence case, albeit with a degree more of ingenuity. Evidently the convergence case is completely covered by statement (\ref{eqn:convergence}). 
\begin{thm}
	\label{thm:DSCounter}
	For any pair of reals $i$, $j$ satisfying (\ref{eqn:ijconditions}) and any $\mathcal{D}$-adic sequence 
	there exists an approximating function $\Psi:\mathbb{N}\rightarrow\mathbb{R}_{\geq0}$ for 
	which
	\[
		\lambda_1(W_{\mathcal{D}}(i, j, \Psi)) = 0  \quad \text{ but } \quad 
		\sum_{r \in \mathcal{A}_\Psi} \Psi^j(r) = \infty.
	\]
\end{thm}
\noindent
Our counterexample is constructed via direct modification of the method of Duffin \& Shaeffer and for completion we include a full exposition in 
\S\ref{sec:ConstructionOfCounterexample}.

%However, our example also shows that the `natural' mixed analogue of the Duffin-Schaeffer Conjecture is not true in general. 

With regards to a mixed simultaneous analogue of the Duffin-Schaeffer Conjecture, one begins by imposing a coprimality condition on the rational approximations as before. Here, this equates to considering the set $W'_{\mathcal{D}}(i, j, \psi)$ of points $x \in [0,1)$ for which the conditions
\begin{equation*}
	\max \left\{\infabs{q}_{\mathcal{D}}^{1/i}, \infabs{qx-p}^{1/j} \right\} \leq  \psi(q), \quad 
	\quad \quad (p, q)=1,
\end{equation*}
are satisfied for infinitely many $p \in \mathbb{Z}$ and $q\in \mathbb{N}$. It should now be obvious that it would be naive to propose that we have
	\[
		\lambda_1(W'_{\mathcal{D}}(i, j, \psi)) = 1  \quad \text{ if } \quad 
		\sum_{r =1}^{\infty} \frac{\varphi(r)}{r}\, \psi(r) = \infty
	\]
for any pair of reals $i$, $j$ satisfying (\ref{eqn:ijconditions}), any approximating function $\psi$ and any $\mathcal{D}$-adic sequence with bounded ratios.
Indeed, it is not difficult to see this statement is false; the function $\psi_1$ above provides a counterexample. A more astute and natural proposal for a mixed Duffin-Schaeffer Conjecture is that
	\[
		\lambda_1(W'_{\mathcal{D}}(i, j, \psi)) = 1  \quad \text{ if } \quad 
		\sum_{r \in \mathcal{A}_\psi} \frac{\varphi(r)}{r}\, \psi^j(r) = \infty.
	\]
%This choice is supported by the discussion in \S\ref{sec:EquivalenceOfSums}.
The function $\psi_1$ certainly does not contradict this statement as the set $\mathcal{A}_{\psi_1}$ 
is empty in this case. Along with its classical counterpart, a proof (or disproof) of this statement remains out of reach.
%That said, we demonstrate that this statement is also false. To be precise,
%the following is proven.

%We prove this theorem via similar methods to those utilised by Duffin and Schaeffer. 
%We remark that since 
%\[
%	W'_{\mathcal{D}}(i, j, \Psi) \subset W_{\mathcal{D}}(i, j, \Psi)\quad \text{ and } \quad
%	\sum_{r \in \mathcal{A}_\Psi} \frac{\varphi(r)}{r} \, \Psi^j(r) \, \, \leq \, \, 
%	\sum_{r \in \mathcal{A}_\Psi}\, \Psi^j(r)
%\] 
%the approximating function constructed in Theorem \ref{thm:DSCounter} both proves the necessity of the monotonicity assumption in Theorem \ref{thm:divergence} and disproves the more `natural' proposal for a mixed Duffin-Schaeffer Conjecture. 

\subsection{Proof of Theorem \ref{thm:DSCounter}}
\label{sec:ConstructionOfCounterexample}

%We will construct an approximating function $\Psi$ for which 
%$\lambda_1(W'_{\mathcal{D}}(i, j, \Psi))$ = $\lambda_1(W_{\mathcal{D}}(i, j, \Psi)) = 0$ 
%but 
%\[
%	\sum_{r \in \mathcal{A}(\Psi)} \Psi^j(r) \quad = \quad 
%	\sum_{r \in \mathcal{A}(\Psi)} \frac{\varphi(r)}{r}\, \Psi^j(r) \quad = \quad \infty.
%\]
%We proceed as follows.
In the vein of Duffin \& Schaeffer we first show for any $R \geq 1$ and  $\epsilon>0$ that there exists an approximating function 
$\psi$ such that
\begin{equation*}
	\sum_{r \in \mathcal{A}_\psi} \psi^j(r) \, \, > \, \, 1, \quad \quad 
	\psi(r)=0 \quad \text{ when } \quad r \leq R,
\end{equation*}
but the set of $x \in (0,1)$ such that 
\begin{equation}
	\label{eqn:admissible}
	\norm{qx} \, < \, \psi^j(q) \quad \quad \text{ for some } q \in \mathcal{A}_\psi,
\end{equation}
has Lebesgue measure strictly less than $\epsilon$.

Let $\alpha$ be a positive number such that $0 <  \alpha < \epsilon/2$ and choose primes 
$p_1, p_2,  \ldots, p_s$ with $p_t > R$ $ (t=1, \ldots, s)$ for some natural number~$s$ to be specified later. 
Since $\mathcal{D}$ has bounded ratios we may choose the primes in such a way that $(p_t, n_k)=1$ for all $t$ and $k$.
Next, let 
\[
	K=K(s, \alpha):=\min\left\{k \in \mathbb{N}: n_k \geq (p_1\cdots p_s/\alpha)^{i/j}  \right\}
\] 
and set 
\[
	N:= n_K p_1\cdots p_s.
\] 
Finally, define
\begin{equation*}		
	\psi(q):  \quad = \quad
	\begin{cases}
		(q\alpha/N)^{1/j}, & n_K \,| \,q, \quad  q\,|\,N, \quad q\neq n_K. \\
		& \\
	0, & \text{otherwise}.
 	\end{cases}
\end{equation*}

We claim that $\psi$ satisfies the desired properties. Let $A_q \subset (0,1)$ denote the set consisting of the $q-1$ open intervals of length $2 \psi^j(q)/q$ with centres at the rationals $p/q$ ($p=1,\ldots, q-1$) and the open intervals $(0, \psi^j(q)/q)$ and $(1-\psi^j(q)/q, 1)$. For small enough $\epsilon$ these intervals are disjoint and so the 
Lebesgue measure of $A_q$ is given by
$2\psi^j(q)=2q\alpha/N$ if $1 \leq q \leq N$. Furthermore, since $\psi^j(q)/q = \alpha/N = \psi^j(N)/N$ we have
\[
	A_N= \bigcup_{\substack{q \, | \, N: \\ n_K \, | \, q \\ q \neq n_K}} A_q
\]
and for all $q$ in this union
\[
	\infabs{q}_{\mathcal{D}} \,\, \leq \,\, \frac{1}{n_K} \,\, \leq \,\, \left(\frac{\alpha}{p_1\cdots p_s}\right)^{i/j} \,\, 
	= \, \, \left(\frac{n_K \alpha }{N}\right)^{i/j} < \,\, \left(\frac{q\alpha}{N}\right)^{i/j} \,\, = \,\, \psi^i(q);
\]
i.e., $q \in \mathcal{A}_\psi$. Hence, as $\psi(q)=0$ for all $q$ not in the union we may deduce that property (\ref{eqn:admissible}) will be satisfied by 
irrational $x \in (0,1)$ if and only if $x \in A_N$. Moreover, we have  $\lambda_1(A_N) = 2\alpha<\epsilon$. 

All that remains is to show that
\[
	\sum_{r \in \mathcal{A}_\psi} \psi^j(r) \quad > \quad 1.
\]
Via the change of variables $\ell:=rn_K^{-1}$ and $M:=Nn_K^{-1}$ we have
\begin{equation*}
	\sum_{r \in \mathcal{A}_\psi}  \psi^j(r) \quad = \quad \frac{\alpha}{N} 	\sum_{\substack{r \, | \, N: \\ n_K \, | \, r \\ r \neq n_K}} r \quad 
	= \quad	\frac{\alpha}{M}  \sum_{\substack{\ell \, > \, 1: \\ \ell \, | \, M}} \ell.
\end{equation*} 
%since $n_K$ and the divisors of $M$ are pairwise coprime. It is readily verified 
%that the arithmetic function $f$ defined by
%\[
%	f(n) \: = \: \sum_{\ell \, | \, n}	\frac{\varphi(\ell)}{\ell^{1-j}}
%\]
%is multiplicative. Therefore, 
%\[
%	f(n) \, = \, \prod_{t=1}^m \left(1+ \frac{\varphi(q_t)}{q_t^{1-j}} 
%	+ \frac{\varphi(q_t^2)}{q_t^{2(1-j)}}+ \cdots +
%	\frac{\varphi(q_t^{\alpha_t})}{q_t^{\alpha_t(1-j)}} \right),
%\]
%where $n=q_1^{\alpha_1}\cdots q_m^{\alpha_m}$ is the unique prime factorization of $n$ (see \cite{Hua}, for example). 
%It 
%follows from the assumption that $\mathcal{D}$ has bounded ratios that the quantity 
%$\varphi(n_K)/n_K$ is bounded below by some positive constant, $\kappa>0$ say, which depends only 
%upon $\mathcal{D}$. 
We may now follow the original argument of Duffin \& Schaeffer. Choose $s$ large enough so that
\begin{equation*}
	\label{eqn:klarge}
	\prod_{t=1}^{s} \left(1+ 1/p_t\right) \quad > \quad 1 \, + \,  1/\alpha.
\end{equation*}
This is always possible because the above product diverges when extended over all primes.
Then, since $M=p_1 \cdots p_s$ it follows by standard arithmetic techniques that 
\begin{eqnarray}
	\frac{\alpha}{M}  \sum_{\substack{\ell \, > \, 1: \\ \ell \, | \, M}} \ell  & = & 	
	\frac{\alpha}{M}  \left(\prod_{t=1}^{s} \left(1+ p_t \right)-1\right) \nonumber \\
	& > & \alpha \left(\prod_{t=1}^{s} \left(1+ 1/p_t\right)-1\right) \nonumber\\
	& > & 1, \nonumber \
\end{eqnarray}
as required. To complete the construction of our counterexample we proceed as follows. Let $\psi_1$ satisfy the above 
properties with $R=R_1:=1$ and $\epsilon=\epsilon_1:=2^{-2}$. 
Then, for some $R_2$ we must have that $\psi_1(q)=0$ for 
all $q > R_2$. Let $\psi_2$ satisfy the above properties with $R=R_2$ and 
$\epsilon=\epsilon_2:=2^{-2}$. Continue in this fashion to choose numbers $R_t$ and construct functions $\psi_t$ satisfying the above properties for $R=R_t$ and $\epsilon=\epsilon_t:=2^{-t}$. Then, define
\begin{equation*}		
	\Psi(q):  \quad = \quad
	\begin{cases}
		\psi_1(q), & q \leq R_2. \\
		& \\
	\psi_t(q), & R_t < q \leq R_{t+1}, \quad t\geq 2.
 	\end{cases}
\end{equation*}
It is clear that
\begin{equation*}
	\sum_{r \in \mathcal{A}_\Psi} \, \Psi^j(r) 
	\quad = \quad \infty,
\end{equation*}
but for $x\in(0,1)$ the system
\begin{equation*}
	\norm{qx} \, < \, \Psi^j(q), \quad \quad \quad q \in \mathcal{A}_\Psi, 
	\quad \quad \quad q > R_t
\end{equation*}
can be satisfied only if $x$ belongs to a set of measure at most
\begin{equation*}
	\sum_{r=t}^{\infty} 2^{-r} \quad = \quad 2^{-t+1},
\end{equation*}
as desired.

\section{Ubiquitous Systems}
\label{sec:ubiquity}

Ubiquity is a fundamental tool for establishing measure theoretic statements in Diophantine approximation and will be utilised in the proof of Theorem \ref{thm:Hdivergence}. This section comprises of a brief description of a restricted form of ubiquity tailored to our needs.

The concept of ubiquitous systems was first introduced by Dodson,
Rynne \& Vickers in \cite{DRV} as a method of determining lower bounds 
for the Hausdorff dimension of limsup sets. Recently, this idea was
developed by Beresnevich, Dickinson \& Velani in \cite{limsup} to
provide a very general framework for establishing the Hausdorff
measure of a large class of limsup sets. A slightly more simplified account is presented in
\cite{limsup2}.

\subsection{The ubiquity setup}
\label{sec:Theubiquitysetup}

Let $(\Omega, d)$ be a compact metric space supporting a non-atomic probability measure $\mu$ and assume that any open subset of $\Omega$ is $\mu$-measurable.
Throughout, $B(c, r)$ will denote a ball in $\Omega$ centred at a point $c$ and of radius $r>0$. The following regularity condition will be imposed on the measure of balls: There exist positive constants $a,b, \delta$ and $r_0$ such that for any $c \in \Omega$ and $r \leq r_0$
\[
ar^{\delta} \, \leq \, \mu(B(c, r)) \, \leq \, br^{\delta}.
\]
If this power law holds then we say $\mu$ is \textit{$\delta$-Ahlfors regular}. It easily follows that if $\mu$ is $\delta$-Ahlfors regular then $\dim\Omega=\delta$ and that $\mu$ is comparable to $\delta$-dimensional Hausdorff measure $\mathcal{H}^{\delta}$. For details see Chapter~$4$ of \cite{Fal}.

Let $\mathcal{R}=\left\{R_{a} \in \Omega: a \in J \right\}$ be a collection of 
points $R_{a}$ in $\Omega$ indexed by some infinite, countable set $J$. The points 
$R_{a}$ are referred to as the \textit{resonant points}. Next, let $\beta: J \rightarrow 
\mathbb{R}_{>0}: a \mapsto \beta_{a}$ be a positive function defined on $J$ for which the number of $a \in J$ with $\beta_{a}$ bounded above is always finite. The latter condition imposed on the function $\beta$ is very natural and is often referred to as the \textit{Northcott property} in reference to Northcott's famous result that the number of algebraic numbers of bounded degree and bounded height is finite.
Finally, given an approximating function $\Psi$ define 
\[
	\Lambda(\Psi):= \left\{ x \in \Omega: x \in B(R_{a}, \Psi(\beta_{a})) 
	\text{ for infinitely many } a \in J \right\}.
\]
It is the measure of this set in which we are interested.
 
To demonstrate the `limsup' nature of $\Lambda(\Psi)$ first choose any two positive 
increasing sequences $l:=\left\{l_k \right\}$ and $u:=\left\{u_k\right\}$ 
such that $l_k < u_k$ and $\lim_{k\rightarrow\infty}l_k=\infty$. These sequences will be referred to as the 
\textit{lower} and \textit{upper} sequences respectively. For $k \in \mathbb{N}$ let 
\[
	\Lambda_l^u(\Psi, k):= \bigcup_{a \in J_l^u(k)} 
	B(R_{a}, \Psi(\beta_{a})),
\]
where $J_l^u(k):=\left\{a \in J: l_k < \beta_{a} \leq u_k \right\}$. Then, it is easily seen that 
\[
	\Lambda(\Psi)= \bigcap_{m=1}^{\infty}\bigcup_{k=m}^{\infty} \Lambda_l^u(\Psi, k).
\]

We can now define what it means to be a ubiquitous system. Let $\rho: \mathbb{R}_{>0} 
\rightarrow \mathbb{R}_{>0}$ be any function with $\rho(r) \rightarrow 0$ as $r \rightarrow 
\infty$ and let
\[
	\Delta_l^u(\rho, k):= \bigcup_{a \in J_l^u(k)} 
	B(R_{a}, \rho(u_k)).
\]
\textbf{Definition (Local $\mu$-ubiquity)} Let $B=B(c,r)$ be an arbitrary ball in $\Omega$ of radius $r \leq r_0$. Suppose there exists a function $\rho$, sequences $l$ and $u$ and 
an absolute constant $\kappa>0$ such that
\begin{equation}
	\label{eqn:ubiquity}
	\mu(B \cap \Delta_l^u(\rho, k)) \, \geq \, \kappa \mu(B) \quad \quad \forall \, k \geq k_0(B).
\end{equation}
Then the pair $(\mathcal{R}, \beta)$ is said to be a \textit{local $\mu$-ubiquitous system 
relative to} $(\rho, l, u)$. The function $\rho$ is referred to as the \textit{ubiquitous function}. %Also, as is noted in~\cite{limsup}, the appearance of the lower sequence $l$ is in the above definition is irrelevant. Indeed, to establish inequality (\ref{eqn:ubiquity}) it suffices to show
%\begin{equation}
%	\label{eqn:ubiquity2}
%	\mu\left(B \cap \bigcup\ _{a \in J^u(k)} \,
%	B(R_{a}, \rho(u_k))\right) \, \geq \, \kappa \mu(B) \quad \quad \forall \, k \geq k_0(B),
%\end{equation}
%where $J^u(k):=\left\{a \in J:  \beta_{a} \leq u_k \right\}$. 

Finally, a function $h$ is said to be \textit{$u$-regular} if there exists a strictly positive constant $\lambda<1$, which may depend on $u$,  such that for $k$ sufficiently large
\[
	h(u_{k+1}) \, \leq \, \lambda h(u_k).
\]

%\subsection{The ubiquity results}
%\label{sec:TheUbiquityResults}

We now present two powerful results associated with ubiquitous systems, which have been tailored to our needs. The first theorem (see \cite[Corollary~2]{limsup}) concerns the $\mu$-measure of the limsup set $\Lambda(\Psi)$ and in our setup corresponds to the case $`s=1'$ in Theorem~\ref{thm:divergence}. The second (see \cite[Theorem~10]{limsup3}) deals with the $s$-dimensional Hausdorff measure $\mathcal{H}^s$ of $\Lambda(\Psi)$ for $0<s<1$. Due to the nature of the 
ubiquity framework it is necessary to deal with the two scenarios separately.

%The following theorem (see \cite[Corollary~2]{limsup}) concerns the $m$-measure of the limsup set $\Lambda(\Psi)$ and is tailored to our needs.

\begin{thmBDV1}
	Let $(\Omega, d)$ be a compact metric space equipped with a $\delta$-Ahlfors regular measure $\mu$. 
	Suppose that $(\mathcal{R}, \beta)$ is a local $\mu$-ubiquitous 
	system relative to $(\rho, l, u)$ and that $\Psi$ is a decreasing approximation function. 
	Furthermore, suppose that either $\Psi$ or $\rho$ is $u$-regular and that 
	\begin{equation*}
%		\label{eqn:conditionBDV1}
		\sum_{k=1}^{\infty} \left( \frac{\Psi(u_k)}{\rho(u_k)} \right)^{\delta} 
		\quad = \quad \infty.
	\end{equation*}
	Then,
	\[
		\mu\left( \Lambda(\Psi) \right) \quad = \quad  1.
	\]
\end{thmBDV1}

\begin{thmBDV2}
	Let $(\Omega, d)$ be a compact metric space equipped with a $\delta$-Ahlfors regular measure $\mu$.	Suppose that $(\mathcal{R}, \beta)$ is a local $\mu$-ubiquitous 
	system relative to $(\rho, l, u)$ and that $\Psi$ is a decreasing approximation function. 
	Furthermore, suppose that $0<s<\delta$. Let $g$ be the positive function given by 
	$g(r):=\Psi^s\rho^{-\delta}$ and let $G:=\limsup_{k\rightarrow \infty}g(u_k)$.
	\begin{itemize}
		\item[(i)] Suppose that $G=0$ and $\Psi$ is $u$-regular. Then,
			\[
				\mathcal{H}^s(\Lambda(\Psi)) = \infty \quad \quad \text{ if } \quad\quad 
				\sum_{k=1}^{\infty}g(u_k) = \infty.
			\]
		\item[(ii)] Suppose that $0<G<\infty$. Then, $\mathcal{H}^s(\Lambda(\Psi)) = \infty$.
	\end{itemize}
\end{thmBDV2}

Before proceeding, we recall a generalisation of the Cauchy condensation test attributed to Oscar Schl\"{o}milch, which can be found 
in \cite[Theorem $2.4$]{Schl}. We will appeal to this result multiple times in our proof.

\begin{thmSchl}
 	Let $\sum_{r=0}^{\infty}a_r$ be an infinite real series whose terms are positive and decreasing and let 
 	$m_0 < m_1 < \cdots$ be a strictly increasing sequence of positive integers for which there exists a constant $M>0$ such that
 	\begin{equation}
 		\label{eqn:Schl}
 		\frac{m_{k+1}-m_k}{m_k-m_{k-1}} \, \leq \, M \quad \text{ for every } k \in \mathbb{N}.
 	\end{equation}
 	Then the series $\sum_{r=0}^{\infty}a_r$ converges if and only if the series 
 	$\sum_{k=0}^{\infty}(m_{k+1}-m_k)a_{m_k}$ converges.
\end{thmSchl}
It should be noted that, taking $m_k=n_k$ for some $\mathcal{D}$-adic sequence $\left\{n_k \right\}$,
condition (\ref{eqn:Schl}) is satisfied for some $M\geq2$ if and only if the sequence $\mathcal{D}$ has bounded ratios.

\section{Proof of Theorem \ref{thm:Hdivergence}}
\label{sec:ProofOfTheoremRefThmDivergence}

For the divergence part of Theorem \ref{thm:Hdivergence} we will appeal to the ubiquity framework described in the previous section. The convergence part follows by well-known arguments stemming from the Borel-Cantelli Lemma. For completeness we include a short  proof here. 
%Throughout 
%we assume that $\mathcal{D}$ is bounded by $M\geq2$; i.e., $n_{k+1}/n_k \leq M$ for all $k \in \mathbb{N}$.

%For each $s$ with $i<s \leq 1$ let $\mathcal{H}^s$ denote $s$-dimensional Hausdorff measure and 
%Assume 
%that the sum $\sum_{r\in \mathbb{N}} r^{1-s} \,\psi^{i+js}(r)$ converges 
%The case `$s=1$'  corresponds to the setting of Theorem \ref{thm:divergence}, where $\mathcal{H}^1$ is comparable to one-dimensional Lebesgue %measure.  

Firstly, note that we may assume $\psi(r) < 1$ for all sufficiently large $r$, for otherwise the sum $\sum_{r\in \mathbb{N}} f_{\psi}(r)$ would surely diverge. So, for each $k \in \mathbb{N}$ sufficiently large we can find a unique natural number $m_k$ for which 
\begin{equation}
	\label{eqn:psicontrol}
	\frac{1}{n_{m_k}} \quad < \quad \psi^i(n_k) \quad \leq \quad \frac{1}{n_{m_k-1}}.
\end{equation} 
This is possible since $\psi$ is decreasing and $\mathcal{D}$ is an increasing sequence. The pseudo-absolute value is discrete, 
so for sufficiently large $k \in \mathbb{N}$ it follows from (\ref{eqn:psicontrol}) that
	\begin{eqnarray*}
		\# \left\{  q \in (n_k, n_{k+1}]: \, q \in \mathcal{A}_\psi \right\} 
		& \leq & \# \left\{  q \in (n_k, n_{k+1}]: \, 
		\infabs{q}_\mathcal{D} < \psi^i(n_k) \right\}\\
		& = & \# \left\{   q \in (n_k, n_{k+1}]: \,  \infabs{q}_\mathcal{D} \leq \frac{1}{n_{m_k}} \right\} \\
		& = & \# \left\{   q \in (n_k, n_{k+1}]: \,  n_{m_k} | \, q  \right\}\\
		& = & \frac{n_{k+1}-n_k}{n_{m_k}}\\
		& < & (n_{k+1}-n_k)\psi^i(n_k). \
	\end{eqnarray*}

Next, for each $q \in \mathcal{A}_\psi$ let $W_q$ denote the set of real numbers $x \in (0,1)$ satisfying 
\begin{equation*}
	%\label{eqn:wellapp}
	\max \left\{\infabs{q}_{\mathcal{D}}^{1/i}, \norm{qx}^{1/j} \right\} \quad < \quad \psi(q)
\end{equation*}
and let $M \geq 2$ be an upper bound for the ratios of consecutive elements of $\mathcal{D}$; i.e., $n_{k+1}/n_k \leq M$ for all $k \in \mathbb{N}$. 
Each set $W_q$ is covered by the $q-1$ open intervals of length $2 \psi^{j}(q)/q$ with centres at the rationals $p/q$ ($p=1,\ldots, q-1$) and the two open intervals $(0, \psi^j(q)/q)$ and $(1-\psi^j(q)/q, 1)$. Let us denote by $E_q$ this collection of covering intervals. For any $k_0 \in \mathbb{N}$ we have that the countable collection
$$\bigcup_{\substack{q \in \mathcal{A}_\psi \\q \, > \, n_{k_0} }} E_q$$
is a $\rho-$cover for  $W_{\mathcal{D}}(i, j, \psi)$ for $\rho=2\psi^j(n_{k_0})/n_{k_0}$.
Thus, the value $\mathcal{H}^s_{\rho}(W_{\mathcal{D}}(i, j, \psi))$ is at most 
	\begin{eqnarray}
		\nonumber 2^s\sum_{\substack{q \in \mathcal{A}_\psi \\q \, > \, n_{k_0} }} 
		q^{1-s}\psi^{js}(q) & 
		\leq & 2^sM^{1-s}\sum_{k=k_0}^{\infty} n_{k}^{1-s}\psi^{js}(n_k)\,\,\sum_{\substack{q \in \mathcal{A}_\psi \\ 
		q \in (n_k, \, n_{k+1}]}} 1\\
		& < & 2^sM^{1-s}\sum_{k=k_0}^{\infty} 
		(n_{k+1}-n_k)n_k^{1-s}\psi^{i+js}(n_k). \label{eqn:bound}\
%		& < & \epsilon,\		
	\end{eqnarray}	
	However, the function $f_{\psi}=r^{1-s}\psi^{i+js}(r)$ is assumed decreasing and $\mathcal{D}$ is assumed to have bounded ratios and so we may apply Schl\"{o}milch's Theorem. The sum $\sum_{r=1}^{\infty} f_{\psi}(r)$ converges so we may take (\ref{eqn:bound}) to be as small as we wish. In particular, as $\rho \rightarrow 0$ (or equivalently as $k_0 \rightarrow \infty$) we have $\mathcal{H}^s_{\rho}(W_{\mathcal{D}}(i, j, \psi)) \rightarrow 0$ and the `convergence' part of Theorem~\ref{thm:Hdivergence} is complete.

We now demonstrate how the ubiquity framework can be applied to the set 
$W_{\mathcal{D}}(i, j, \psi)$. 
Firstly, choose a natural number~$c$. It is  easy to see that $W_{\mathcal{D}}(i, j, \psi)$ can be expressed in the form 
$\Lambda(\Psi)$ with
\[
	\Omega:= [0,1], \quad \Psi(r):= \psi^j(r)/r, \quad J:=\left\{(p,q) \in 
	\mathbb{N} \times \mathbb{N}: q \in \mathcal{A}_\psi, \, 0 \leq p \leq q \right\}, 
\]
\[
	a:=(p,q) \in J, \quad \beta_{a}:=q, \quad R_{a}:= p/q, \quad 
	u_k:=l_{k+1}:=n_{ck}, \quad \mu:=\lambda_1, \quad \delta:=1,
\]
\[
	J_l^u(k):= \left\{(p,q) \in J: n_{c(k-1)} < q \leq n_{ck} \right\}, 
	\quad \Lambda_l^u(\Psi, k):= \bigcup_{(p,q) \in J_l^u(k)} B(p/q, \psi^j(q)/q),
\]
so that
\[
	W_{\mathcal{D}}(i, j, \psi) \, \, = \, \, \limsup_{k\rightarrow \infty} \Lambda_l^u(\Psi, k).
\]
%It is clear that $\lambda_1$ is $\delta$-Ahlfors regular and that $d$ is the standard 
%Euclidean metric $d(x,y):=\infabs{x-y}$. 
The natural number $c$ above is introduced for technical reasons and its appearance will be qualified later, suffice to say we may not take $c=1$. 

We now show that this system is locally $\lambda_1$-ubiquitous relative to $(\rho, l, u)$,
 for $l$ and $u$ as chosen above and some real positive function $\rho$ satisfying with $\rho(r) \rightarrow 0$ as $r \rightarrow \infty$. 
 It is apparent 
that an appropriate choice of ubiquitous function is $\rho(q):=\gamma/q^2 \psi^i(q)$ 
for some constant $\gamma>0$ for then the sum 
\begin{equation*}
	\sum_{k=1}^{\infty} \left( \frac{\Psi(u_k)}{\rho(u_k)} \right)^{\delta} \quad = \quad 
	\sum_{k=1}^{\infty} \frac{n_{ck}^2\psi^i(n_{ck})\psi^j(n_{ck})}{\gamma \, n_{ck}} \quad = \quad 
	\frac{1}{\gamma}\sum_{k=1}^{\infty} n_{ck}\psi(n_{ck})
\end{equation*}
diverges if and only if the sum $\sum_{r=1}^{\infty} \psi(r)$ diverges 
by the result of Schl\"{o}milch. 

Next, we point out an important observation. When $\sum_{r \in \mathbb{N}} r^{1-s}\psi^{i+js}(r)=\infty$ 
and $s \in (i,1]$ we may assume that 
\begin{equation}
	\label{eqn:assumption}
	\psi^i(r) \, > \, 1/r \quad \text{ for all } r \in \mathbb{N}. 
\end{equation}
To see this, let $\mathcal{R}:=\left\{r_k\right\}_{r \in \mathbb{N}}$ be an increasing  sequence of integers for which $\psi^i(r_k)\leq1/r_k$. Then, 
for $s \in (i, 1]$ we have
\[
	\sum_{k \in \mathbb{N}} r_k^{1-s}\psi^{i+js}(r_k) \, \leq \, \sum_{k \in \mathbb{N}} 
	r_k^{-(1+j/i)s} < 	\infty \quad \text{ and } \quad \sum_{r \, \in \, \mathbb{N} \setminus \mathcal{R}} 
	r^{1-s} \psi^{i+js}(r) = \infty.
\]
Moreover, for each $k \in \mathbb{N}$ we have
\[
	\psi^i(r_k) \, \leq \, \frac{1}{r_k} \, \leq \, \infabs{r_k}_{\mathcal{D}}
\]
and so $r_k \notin \mathcal{A}_\psi$.  The upshot is that we may choose
$J \subset \mathbb{N} \times (\mathbb{N}\setminus \mathcal{R})$ in the ubiquity setup and neither the set $W_{\mathcal{D}}(i, j, \psi)$ nor the divergence of the corresponding 
volume sum is affected by the removal of the integers $r_k$. 

Observation (\ref{eqn:assumption}) immediately implies $\rho(r) \rightarrow 0$ as $r \rightarrow \infty$ as required in the ubiquity setup. Furthermore, let $M \geq 2$ be an upper bound for the ratios of consecutive elements of $\mathcal{D}$. Then, the monotonicity of $\psi$ immediately implies that
\[
	\frac{\psi^j(n_{c(k+1)})}{n_{c(k+1)}} \, \leq \, \frac{\psi^j(n_{ck})}{n_{c(k+1)}} \, \leq \, \frac{\psi^j(n_{ck})}{M^cn_{ck}}
\]
and so $\Psi$ is trivially $u$-regular. 
%NEED TO SHOW WHY $\rho$ TENDS TO ZERO AND $\Psi$ IS $U$-REGUALAR AND WHY THEN $G=0$!!
Hence, to prove the `divergence' part of Theorem~\ref{thm:Hdivergence} 
it suffices to show the following holds.
\begin{prop}
 	\label{prop:ubiquity}
	Let $\rho(q):=\gamma/q^2 \psi^i(q)$ for some $\gamma>0$. Then, the system defined above is a locally $\lambda_1$-ubiquitous  relative to the triple $(\rho, n_{c(k-1)}, n_{ck})$ for some $c \in \mathbb{N}$. 
\end{prop}

We begin by modifying the sequence specified in 
(\ref{eqn:psicontrol}). Once more we may assume that $\psi(r) < 1$ for large $r$ and so for any sufficiently large $k \in \mathbb{N}$ and any $c \in \mathbb{N}$ we can find a unique natural number $m_k:=m_k(c)$  for which 
\begin{equation}
	\label{eqn:cpsicontrol}
	\frac{1}{n_{cm_k}} \quad < \quad \psi^i(n_{ck}) \quad \leq \quad \frac{1}{n_{c(m_k-1)}}.
\end{equation} 
To prove Proposition~\ref{prop:ubiquity} we  require 
%two further results. The first describes an inherant 
%property of monotonic approximating functions with divergent sums and the second is a 
the following consequence of a classical theorem of Dirichlet. 
\begin{prop}
		\label{prop:Dirichlet}
		Fix $c \in \mathbb{N}$. Then, for every $x \in \mathbb{R}$ and every $k \in \mathbb{N}$ 
		there exists $p/q \in \mathbb{Q}$	with $n_{cm_k} \leq q \leq n_{ck}$ such that 
		\begin{equation}
			\label{eqn:Dirichlet}
			\infabs{x-\frac{p}{q}} \quad < \quad \frac{n_{cm_k}}{qn_{ck}} \quad \quad \text{ and } \quad \quad 
			\infabs{q}_\mathcal{D} \quad \leq \quad \frac{1}{n_{cm_k}}.
		\end{equation}
	\end{prop}
	\begin{proof}[Proof of Proposition \ref{prop:Dirichlet}]
		Dirichlet's theorem states that for all $x' \in \mathbb{R}$ and for all $N \in \mathbb{N}$ there exists $p/q' \in \mathbb{Q}$ with $q' \leq N$ such that
		\[
			\infabs{x'-p/q'} \quad < \quad 1/q'N.
		\]
 		Let $N:=n_{ck}/n_{cm_k}$. Observation (\ref{eqn:assumption}) guarantees that $N\geq1$. Next, set $x:=x' n_{cm_k}$  
 		and $q=n_{cm_k}q'$. Then, for all $x\in \mathbb{R}$ we have
 		\[
			\infabs{xn_{cm_k}-\frac{pn_{cm_k}}{q}} \quad < \quad \frac{n_{cm_k}^2}{qn_{ck}} %\quad \stackrel{(\ref{eqn:control})}{\leq} 
			%\quad \frac{n_{m_k}d_{m_k}}{qn_k\psi^i(n_k)},
		\]
 		whereby upon division by $n_{cm_k}$ the desired inequality is reached. Furthermore, $n_{cm_k} \leq q\leq n_{cm_k}n_{ck}/n_{cm_k} = n_{ck}$ and $\infabs{q}_\mathcal{D} \leq 1/n_{cm_k}$ as required.
	\end{proof}

	In what follows, for $r \in \mathbb{N}$ we denote by $K^-(r)$ the set of $q \in \mathbb{N}$ with $\infabs{q}_{\mathcal{D}} \leq 1/n_{cm_k}$ for which $q \leq n_{c(r-1)}$, whereas $K^+(r)$ will denote the set of
	$q \in \mathbb{N}$ with $\infabs{q}_{\mathcal{D}} \leq 1/n_{cm_k}$ that satisfy $n_{c(r-1)} < q \leq n_{cr}$.
	Recall that $\rho(r):=\gamma/r^2 \psi^i(r)$ for some $\gamma>0$.
	
%	It is clear that condition (\ref{eqn:conditionBDV1}) is fulfilled for
%	\[
%		\sum_{r=1}^{\infty} \psi(r) \; = \; \infty \quad \stackrel{(\ref{eqn:cc})}{\Longleftrightarrow} \quad
%		 \sum_{k=1}^{\infty}n_k \psi(n_k)\; = \; \frac{1}{\delta}\sum_{k=1}^{\infty} \frac{\Psi(n_k)}{\rho(n_k)} \; = \; \infty.
%	\]
%	Furthermore, the function $\Psi$ is $\mathcal{D}$-regular since
%	\[
%		\Psi(n_{k+1}) \; = \; \frac{\psi^j(n_{k+1})}{n_{k+1}} \; = \; \frac{\psi^j(n_{k+1})}{d_{k+1}n_k} \; \leq \; \frac{\psi^j(n_k)}{2n_k} \; = \; \frac{\Psi(n_k)}{2}.
%	\]
%	In view of Theorem~BDV, to prove Theorem \ref{thm:divergence} it suffices to show that $(\mathcal{D}, i, j, \psi)$ induces a local ubiquitous system relative to $\rho$. 

%of \S\ref{sec:Theubiquitysetup},
 To prove Proposition \ref{prop:ubiquity} it now suffices to show that for every interval $I\subset [0,1)$ there exists an absolute 
constant $\kappa>0$ such that
	\begin{equation}
	\label{eqn:sufficient}
	\lambda_1\left(I \cap \bigcup_{\substack{q \in \mathcal{A}_\psi: \\ q \in (n_{c(k-1)}, \, n_{ck}]}} \, \, 
	\bigcup_{p=0}^{q-1}	\: B\left(\frac{p}{q}, \rho(n_{ck})\right)\right) \, \geq \,\kappa \lambda_1(I)
\end{equation}
for all $k$ sufficiently large.  Assume $M\geq2$ is an upper bound for the ratios of consecutive elements of $\mathcal{D}$. 
Upon setting $\gamma=M^{2c}$ it is easily verified that the LHS of (\ref{eqn:sufficient}) is bounded below by 
\begin{equation}
	\label{eqn:reform}
	\lambda_1\left(I \cap \bigcup_{K^+(k)} \:
	\bigcup_{p=0}^{q-1}	\: B\left(\frac{p}{q}, \frac{n_{cm_k}}{q \, n_{ck}} \right) \right).
\end{equation}
To see this simply note that for $n_{c(k-1)} < q \leq n_{ck}$ we have
\[
	n_{ck} \, \, < \, \, q \prod_{t=c(k-1)+1}^{ck} \frac{n_t}{n_{t-1}} \, \, \leq \, \, q M^{ck-(c(k-1) +1) +1} 
	\, \, = \, \, q M^c
\]	
and by definition
\[
	n_{cm_k} \, \,  = \, \, n_{c(m_k-1)} \prod_{s=c(m_k-1)+1}^{cm_k} \frac{n_s}{n_{s-1}} \, \, \leq 
	\, \,  n_{c(m_k-1)} M^c \, \, \leq \, \, \psi^{-i}(n_{ck})M^c.
\]

Proposition \ref{prop:Dirichlet}  implies that the value in $(\ref{eqn:reform})$ exceeds $\lambda_1(I) - \lambda_1(\mathcal{J})$, where
\[
	\mathcal{J}:= \bigcup_{K^-(k)} \:
	\bigcup_{p=0}^{q-1}	\: B\left(\frac{p}{q}, \frac{n_{cm_k}}{q \, n_{ck}} \right).
\]
However, for  each $q$  there are at most 
	$\lambda_1(I)q+3$ possible choices for $p$ and so
	\begin{eqnarray*}
		\lambda_1(\mathcal{J}) & \leq & 2 \sum_{K^-(k)} \:
		\frac{n_{cm_k}}{qn_{ck}} \left( \lambda_1(I)q+3 \right)\\	
		& = & 2\lambda_1(I)\frac{n_{cm_k}}{n_{ck}} \sum_{K^-(k)} \: 1 \quad + \quad 
		\frac{6n_{cm_k}}{n_{ck}} \sum_{K^-(k)} \: \frac{1}{q}.\
	\end{eqnarray*}
	For each $r \in \mathbb{N}$ the cardinality of $K^+(r)$ is bounded above by $(n_{cr}-n_{c(r-1)})/n_{cm_k}$.
	Therefore,
	\begin{eqnarray*}
		\frac{6n_{cm_k}}{n_{ck}} \: \sum_{K^-(k)} \: \frac{1}{q} 
		\quad &\leq & \quad \frac{6n_{cm_k}}{n_{ck}} \: \sum_{r=1}^{k-1}
		\: \sum_{K^+(r)} \frac{1}{q}\\
		& < & \quad \frac{6n_{cm_k}}{n_{ck}} \: \sum_{r=1}^{k-1}
		\: \frac{(n_{cr}-n_{c(r-1)})}{n_{c(r-1)}n_{cm_k}} \\
		& < & \quad \frac{6(M^c-1)(k-1)}{n_{ck}} \\
		& < & \quad \frac{\lambda_1(I)}{4},\
	\end{eqnarray*}
	for $k$ large enough.
	Moreover, the cardinality of $K^-(r)$ is bounded above by 	$n_{c(r-1)}/n_{cm_k}$ for $r \in \mathbb{N}$ and so
	\begin{eqnarray*}
		2\lambda_1(I)\frac{n_{cm_k}}{n_{ck}} \sum_{K^-(k)} \: 1 \quad & \leq & \quad
		2\lambda_1(I)\frac{n_{c(k-1)}}{n_{ck}} \\
		%& \leq & 2\lambda_1(I) \left( \prod_{t=c(k-1)+1}^{ck} \: \frac{n_t}{n_{t-1}} \right)^{-1}\\
		& \leq & 2\lambda_1(I) 2^{-(ck-c(k-1))} \\
		& = & 2^{1-c} \lambda_1(I).\
	\end{eqnarray*}
	It follows that for $c \geq 2$ and for sufficiently large $k$ we have $\lambda_1(\mathcal{J}) \leq 3\lambda_1(I)/4$, 
	and inequality (\ref{eqn:sufficient}) indeed holds with $\kappa=1/4$.	

\subsection*{Acknowledgements}
The authors would like to convey their continuing gratitude to Sanju Velani for introducing the authors to the problems at hand. SH would like to thank the referees for their many valuable and exceptionally useful comments.

\providecommand{\bysame}{\leavevmode\hbox to3em{\hrulefill}}
\providecommand{\MR}{\relax\ifhmode\unskip\space\fi MR }
\providecommand{\MRhref}[2]{%
  \href{http://www.ams.org/mathscinet-getitem?mr=#1}{#2}
} \providecommand{\href}[2]{#2}


\begin{thebibliography}{10}

\bibitem{BLV} D.~Badziahin, J.~Levesley \& S.~Velani, \emph{The mixed Schmidt conjecture in the theory of Diophantine approximation}, Mathematika \textbf{57} (2011) 239--245.

\bibitem{limsup} V.~Beresnevich, D.~Dickinson \& S.~Velani, \emph{Measure Theoretic Laws for limsup Sets}, Mem. Amer. Math. Soc. \textbf{179} (2006), no. 846, 1--91.

\bibitem{limsup3} \bysame, \emph{Diophantine approximation on planar 
curves and the distribution of rational points}, 
Ann. of Math. (2),\textbf{ 166} (2007), 367--426. With an Appendix II by R.C.~Vaughan.

\bibitem{BHHV} V.~Beresnevich, G.~Harman, A.~Haynes \& S.~Velani, \emph{The Duffin-Schaeffer Conjecture with extra divergence II}, submitted. Preprint available at arxiv:1201.1210 (2012).

\bibitem{BHV1} V.~Beresnevich, A.~Haynes \& S.~Velani, \emph{Multiplicative zero-one laws and metric 
number theory}, (in preparation). Preprint available at arxiv:1012.0675 (2011).

\bibitem{BV10} V. Beresnevich \& S. Velani, \emph{Classical metric Diophantine approximation revisited:
The Khintchine-Groshev theorem}, Int. Math. Res. Not. \textbf{2010}, no. 1,  69--86.

%\bibitem{BV12} \bysame, \emph{A note on three problems in Diophantine approximation}, (in preparation), 2011.

\bibitem{limsup2} \bysame, \emph{Ubiquity and a general logarithmic law for geodesics}, S\'eminaires \& Congr\'es \textbf{22} (2009), 21--36.

\bibitem{Bes} A.S.~Besicovitch, \emph{Sets of fractional dimensions (IV): on rational approximation to real numbers}, J. London Math. Soc. \textbf{9} (1934), 126--131.

\bibitem{Schl} D.D.~Bonar \& M.~Khoury Jr, \emph{Real Infinite Series}, 
Mathematical Association of America, Washington DC, 2006.

\bibitem{Bug} Y.~Bugeaud, \emph{An inhomogeneous Jarník theorem}, J. Anal. Math. \textbf{92} (2004), 327--349.

%\bibitem{BugHV} Y.~Bugeaud, A.~Haynes \& S.~Velani, \emph{Metric considerations concerning the mixed
%Littlewood Conjecture,} Int. J. Number Theory (to appear). Preprint available at arXiv:0909.3923 (2009).

\bibitem{DRV} M.M.~Dodson, B.P.~Rynne \& J.A.G.~Vickers, \emph{Diophantine Approximation and a
lower bound for Hausdorff Dimension}, Mathematika \textbf{37} (1990), 59--73.

\bibitem{DS} R.J.~Duffin \& A.C.~Schaeffer, \emph{Khintchine's problem in metric Diophantine approximation},
Duke J. \textbf{8} (1941), 243--255.

\bibitem{Fal} K.~Falconer, \emph{Fractal Geometry (Mathematical Foundations and Applications), Second Edition}, John Wiley \& Sons Ltd, 2004.

\bibitem{Gal}	P.X.~Gallagher, \emph{Metric simultaneous Diophantine aproximations.} J.\ London Math.\ Soc.\ \textbf{37} (1962), 387--390.

%\bibitem{Gal2} \bysame, \emph{Metric simultaneous diophantine approximation. II}, Mathematika \textbf{12} (1962), 123--127.

\bibitem{Gro38} A.~Groshev, \emph{A theorem on a system of linear forms}, Dokl. Akad. Nauk SSSR \textbf{19} (1938), 151--152 (in Russian).

\bibitem{Khi} A.~Khintchine, \emph{Zur metrischen Theorie der diophantischen Approximationen.}, Math. Z. \textbf{24} (1926), 706--714.

\bibitem{Harman} G.~Harman \emph{Metric Number Theory}, LMS Monographs 18, Clarendon Press, Oxford, 1998.

\bibitem{HH} S.~Harrap \& A.~Haynes, \emph{The mixed Littlewood conjecture for pseudo absolute values}, submitted. Preprint available at arxiv:1012.0191 (2011).

\bibitem{Hua} L.K.~Hua, \emph{Introduction to number theory}, Springer-Verlag, Berlin, 1982.

\bibitem{Jar} V.~Jarnik, \emph{Diophantischen Approximationen und Hausdorffsches Mass}, Mat. Sbornik \textbf{36} (1929), 371--382.

\bibitem{Li} Y.~Li, \emph{The winning property of mixed badly approximable numbers}, submitted. Preprint available at arXiv:1212.6584 (2012).

\bibitem{dMT} B.~de Mathan \& O.~Teuli\'{e}, \emph{Problèmes Diophantiens simultan\'{e}s}, Monatsh. Math. \textbf{143} (2004), 229--245.

\bibitem{Sch} W. M.~Schmidt, \emph{A metrical theorem in Diophantine approximation}, Canad. J. Math. \textbf{12} (1960), 619--631.

\end{thebibliography}
\end{document}